\documentclass{amsart}
           
\usepackage{enumerate,amsmath,amsxtra,amssymb}
\usepackage[margin=1.7in]{geometry}

\newtheorem*{Lem}{Lemma}

\newtheorem*{Prop}{Proposition}
\newtheorem*{Cor}{Corollary}
\newtheorem*{Thm}{Theorem}
\newtheorem*{cass}{Casselman's pairing}

\numberwithin{equation}{subsection}

\theoremstyle{remark}
\newtheorem*{Rmk}{Remark}

\newcommand{\spi}{\varepsilon(\pi)}
\newcommand{\spit}{\varepsilon_{\theta}(\pi)}
\newcommand{\spip}{\varepsilon_{ \mathstrut  \theta'}(\pi)}
\newcommand{\spin}{\varepsilon_{\theta}(\pi_N)}
\newcommand{\sptn}{\varepsilon_{\theta}(\tau)}

\newcommand{\pv}{(\pi,V)}

\newcommand{\cx}{\mathbb{C}}

\newcommand{\ov}{\overline{v}} 
\newcommand{\ou}{\overline{u}} 

\newcommand{\pic}{\pi\spcheck}

\newcommand{\tc}{\tau\spcheck}
\newcommand{\cv}{V\spcheck}

\newcommand{\pcv}{(\pi\spcheck,V\spcheck)}

\newcommand{\vform}{\langle\mspace{7mu},\mspace{6mu}\rangle}
\newcommand{\hform}{ (\mspace{7mu},\mspace{6mu})}

\newcommand{\jform}{ (\mspace{7mu},\mspace{6mu})_N}
\newcommand{\sform}{ [\mspace{7mu},\mspace{6mu}]}

\newcommand{\ta}{\theta}
\newcommand{\Ta}{\Theta}

\newcommand{\De}{\Delta}

\newcommand{\dep}{\delta_P}
\newcommand{\depo}{\delta_{\overline{P}}}
\newcommand{\depp}{\delta_P^{1/2}}
\newcommand{\deppo}{\delta_P^{- 1/2}}
\newcommand{\depm}{\delta_P^{- 1/2}}
\newcommand{\depom}{\delta_{\overline{P}}^{- 1/2}}

\newcommand{\pit}{\pi^\theta}
\newcommand{\tat}{\tau^\theta}
\newcommand{\tg}{{}^{\theta}g}

\newcommand{\rpos}{\mathbb{R}^\times_{\texttt{pos}}}
\newcommand{\indp}{\iota_P^G}
\newcommand{\indpb}{\iota_{\overline{P}}^G}

\newcommand{\ompi}{\omega_\pi}

\newcommand{\fas}{\mathfrak{a}^\ast}
\newcommand{\fals}{\mathfrak{a}_L^\ast}
\newcommand{\fams}{\mathfrak{a}_M^\ast}
\newcommand{\fan}{\mathfrak{a}_P^{\ast, -}}

\title{Signs, involutions and Jacquet modules} 
\author{Alan Roche}   
\address{Department of Mathematics \\
         University of Oklahoma\\
         Norman, OK 73019\\
          U.S.A.}
\email{aroche@math.ou.edu}
\author{Steven Spallone}   
\address{School of Mathematics\\
                 TIFR\\
                 Homi Bhabha Rd.\\
                 Mumbai 400 005\\
                 India}
\email{sspallone@gmail.com}

\date{\today}                                        

\begin{document}

\thanks{}


\begin{abstract}
Let $G$ be a connected reductive $p$-adic group and let $\theta$ be an automorphism of $G$ of order at most two. 
Suppose $\pi$ is an irreducible smooth representation of $G$ that is taken to its dual by $\theta$. 
The space $V$ of $\pi$ then carries a non-zero bilinear form $(\mspace{7mu},\mspace{6mu})$, unique up to scaling, 
with the invariance property $(\pi(g)v, \pi({}^{\ta}g)w) = (v,w)$, for $g \in G$ and $v, w \in V$. The form is 
easily seen to be symmetric or skew-symmetric and we  set  
$\varepsilon_\theta(\pi)  = \pm1$ accordingly.  We use Cassleman's pairing (in commonly observed circumstances) to express 
$\varepsilon_\theta(\pi)$ in terms of  certain Jacquet modules of $\pi$ and thus, via the Langlands classification,  reduce the problem of 
determining the sign to the case of of tempered representations. For the transpose-inverse involution of the general linear group, 
we show that the associated signs are always one.
 \end{abstract}

\maketitle


\section*{Introduction}

This is the first in a series of papers devoted to what we call ordinary and twisted signs for reductive $p$-adic groups. 
These signs make sense in a broad setting. To define them, let $G$ be a topological group and let   
$\ta$ be a continuous automorphism of $G$ of order at most two. 
Suppose $\pi$ is an irreducible (complex) representation of $G$ such that $\pit \simeq \pic$ where 
$\pic$ denotes the dual or contragredient of $\pi$ and $\pit$ the $\ta$-twist of $\pi$ given by $\pit(g) = \pi({}^{\ta}g)$, for $g \in G$. 
The underlying space $V$ of $\pi$ then admits a non-degenerate bilinear form such that 
\[
      (\pi(g)v_1, \pit(g)v_2) = (v_1, v_2),  \quad \,\,\, g \in G, \, v_1, v_2 \in V.
 \]
In the presence of Schur's Lemma, the form is unique up to scalars and so is easily seen to be symmetric or skew-symmetric.  
We set 
\[
   \spit = 
          \begin{cases}
   \mspace{13mu} 1 \quad &\text{if the form is symmetric},  \\ 
            -1 \quad &\text{if the form is skew-symmetric}.
     \end{cases}
\]

If $\ta$ has order two, we call $\spit$ the $\ta$-twisted sign of $\pi$.  
If $\ta=1$ (so that $\pi$ is self-dual), we simply write $\spi$ in place of $\varepsilon_1(\pi)$ 
and refer to it as the ordinary sign of $\pi$.

Ordinary and twisted signs have been extensively studied for finite groups of Lie type, largely by examining a
family (or closely related families) of such groups at a time (general linear, special linear, symplectic  etc.) -- see, for example, \cite{Gow, Gow1, P, Vin1}.
There is a small but growing literature on these signs in our setting of smooth representations of 
reductive $p$-adic groups (\cite{P1, Vin, PR, LV}). 

To be more precise, let $F$ be a non-Archimedean local field. Our concern is with involutions and 
associated signs for the groups  of $F$-points of certain connected reductive algebraic groups over $F$.
Many  groups $G$ in this class admit an involution $\ta$ such that $\pit \simeq \pic$, for {\it all} irreducible smooth representations $\pi$
of $G$.  For example, if $G = {\rm GL}_n(F)$ then it is an old result of Gelfand and Kazhdan that the involution 
$g \overset{\ta}{\longmapsto}  {}^{\top}g^{-1}$ (where $\top$ denotes {\it transpose}) has this property \cite{GK}.
We list other examples in \S\ref{examples}.

The two classes of signs, ordinary and twisted, are closely related. For instance, suppose $\tau$ is an irreducible discrete series 
representation of a proper Levi subgroup $M$ of $G$ and suppose that the normalized induced representation $\pi = \indp (\tau)$ is 
irreducible and self-dual (where $P$ is a parabolic subgroup of $G$ with Levi component $M$).
 If $\tau$ is itself self-dual, then it is straightforward to see that $\pi$ and $\tau$ have the same sign. 
In most cases, however, $\tau$ is not self-dual. What holds then is that there is an involution $\ta$ on $M$ such that 
$\tau^\ta \simeq \tc$. Under commonly observed circumstances, our methods lead to a simple explicit relation 
between the signs $\spi$ and $\sptn$. In this way, and in similar but more elaborate ways, the study of ordinary signs on $G$ 
is entwined with the study of twisted signs on Levi subgroups of $G$. 

For much of the paper, we work in a quite general setting. We take an involution $\ta$ on $G$
and a  smooth irreducible representation $\pi$ of $G$ such that $\pit \simeq \pic$.  
We use Cassleman's pairing \cite{Cass} to express the sign $\spit$ in terms of suitable Jacquet modules of $\pi$.  
For this, we need to assume that $\ta$ satisfies certain technical properties known to hold in a broad range of cases
(see \S \ref{setup}). 
In particular, we assume that $\ta$ preserves a Levi component $M$ of a parabolic subgroup $P$ of $G$, more strongly that $\ta$
takes $P$ to its $M$-opposite. 
Writing $N$ for the unipotent radical of $P$ and $\overline{N}$ for the unipotent radical of the $M$-opposite of $P$,  
Casselman's pairing gives a non-degenerate $M$-invariant pairing between the Jacquet modules $\pi_N$ and $(\pic)_{\overline{N}}$.
Using its defining properties, it is straightforward to import the pairing into our setting to obtain the following
descent statement, our principal technical tool. 

\begin{Prop}
With assumptions and notation as above, $(\pi_N)^\ta \simeq (\pi_N)\spcheck$. Moreover, the Jacquet module $\pi_N$
naturally carries a sign $\spin$ (even though it is in general reducible) and $\spit = \spin$.
In particular, if  $\pi_N$ admits an irreducible subrepresentation $\tau$ that occurs with multiplicity one
(as a constituent of $\pi_N$) and is such that $\tau^\ta \simeq \tau\spcheck$, then $\spit = \sptn$. 
\end{Prop}

The hypotheses of the proposition hold in the setting of the Langlands classification. This allows us (under our assumptions) to 
reduce the problem of determining $\ta$-twisted signs to the case of tempered representations. 

The  proposition has implications for ordinary signs. For these, one takes $\ta$ to be an inner automorphism with respect to a suitable element $n$ 
such that $n^2$ is central. Then $\pit \simeq \pic$ simply says that $\pi$ is self-dual and it is easy to see that the ordinary and twisted signs of $\pi$ 
are related by $\spi = \spit \,\ompi(n^2)$ where $\ompi$ denotes the central character of $\pi$ (a special case of (\ref{sign-thetap}) below).
Our descent result then frequently reduces the problem of determining $\spi$ to the case of discrete series representations where often 
powerful other techniques (in particular, global methods as in \cite{PR})
can be brought to bear. We will explore this application of our method elsewhere in the case of certain classical groups. 

In the final sections of the paper, we take $G = {\rm GL}_n(F)$ and set ${}^{\ta}g =  {}^{\top}g^{-1}$, for $g \in G$.  We give two proofs that 
the corresponding twisted signs are always one. 

\begin{Thm}
Suppose $\pi$ is  an irreducible smooth representation of  ${\rm GL}_n(F)$. Then $\spit = 1$.
\end{Thm}
\noindent
This has been proved in many instances by Vinroot \cite{Vin}. Our arguments complete his in that we reduce 
to the case of generic representations (i.e., representations admitting a Whittaker model), a case covered by his methods. 
We choose a somewhat different route, however, and each of our proofs is independent of \cite{Vin}.  
The result is crucial in our approach to the study of ordinary signs for certain classical groups. 
The exact analogue of the theorem also holds in the case of general linear groups over finite fields and the methods of this paper can be adapted
to yield a proof. The result in this case, however, is well-known. It was first proved under a different formulation by Gow 
in odd characteristic (\cite{Gow}~Theorem~4)  and then by MacDonald in arbitrary characteristic (\cite{MacD}~pages~289-90). 
The paper \cite{BG} has a pleasant discussion of the relation between Gow's (and MacDonald's) formulation and twisted signs.
The result is in fact equivalent to the existence and uniqueness of Klyachko models (as follows, for example, from 
 \cite{HZ} which gives yet another proof).

We are grateful to Ryan Vinroot for helpful correspondence. 
  
\section{Preliminaries on signs and involutions}  \label{review}
Let $G$ be a separable locally profinite group and let $\pv$ be a smooth irreducible representation of $G$. 
We write $\pcv$ for the smooth dual or contragredient of $\pv$ and $\vform$ for the canonical non-degenerate
$G$-invariant pairing on $V \times \cv$ (given by evaluation). Let $\theta$ be a continuous automorphism of $G$ of order at most two. 
Define the $\ta$-twist $(\pit, V)$ of $\pv$ by  
\[
\pit(g)\,v = \pi(\tg)\,v,  \quad \,\,\, \, g \in G,  v \in V. 
\]
Suppose that $\pit \simeq \pic$.  This implies that $\pi$ is equivalent to its double dual $\pi\spcheck {}\spcheck$ via the canonical $G$-map 
\begin{equation*}    
   v \mapsto \langle v,  \mspace{6mu} \rangle:(\pi, V) \to  (\pi\spcheck {}\spcheck, V\spcheck {}\spcheck),    \tag{$\ast$}
\end{equation*}
i.e., $\pi$ is admissible. Indeed, $\pi \simeq  (\pic)^\ta$ and so, using $(\pic)^\ta = (\pit)\spcheck$, 
\[
\pi \simeq (\pit)\spcheck \simeq   \pi\spcheck {}\spcheck. 
\]
In particular, $\pi\spcheck {}\spcheck$ is irreducible, and thus the non-zero map ($\ast$) is  an isomorphism.
  
\subsection{} \label{inv-form}
We recall how the sign $\spit$ is attached to $\pi$. 
The same formalism applies in any setting in which Schur's Lemma is available.
 
Let $s:(\pit, V) \to \pcv$ be an isomorphism. Since $G$ is separable, Schur's Lemma applies and 
$s$ is unique up to scaling.  We set 
\begin{equation*}  
                         (v, w) = \langle v, s(w) \rangle, \quad \,\, \forall \, v, w \in V.
\end{equation*}
Then $\hform$ is a non-degenerate bilinear form on $V \times V$. 
It is clearly $G$-invariant in the sense that 
\begin{equation} \label{theta-invariance}
                  (\pi(g)v,\pit(g) w) = (v, w) , \quad \,\, \forall \, g \in G, \,\, \forall \, v, w \in V.
\end{equation}
Moreover, any non-degenerate bilinear form on $V \times V$ with this invariance property arises in this way from a non-zero element
of $\text{Hom}_G(\pit, \pic)$.  Thus any such form is unique up to multiplication by non-zero scalars. 
 
The dual or adjoint $s\spcheck: V \to \cv$ of $s:V \to \cv$ is  characterized by
\begin{equation}  \label{symmetry}
            \langle  v, s\spcheck(w) \rangle = \langle w, s(v) \rangle, \quad \,\, \forall \, v, w \in V.
\end{equation} 
By a straightforward calculation, it again intertwines $\pit$ and $\pic$, and so, by Schur's Lemma, 
$s\spcheck = c \,s$, for some nonzero scalar $c$. Dualizing once more, 
\begin{align*} 
        s  &= (s\spcheck) \spcheck  \\ 
             &= (c \, s)\spcheck  \\
              &= c^2 s, 
 \end{align*}
whence $c^2 = 1$ and $c = \pm 1$.  We set $c = \spit$. It clearly depends only on the equivalence class of $\pi$. 
Rewriting (\ref{symmetry}) in terms of  $\hform$, we have
\[
              (v, w)  = \spit\, (w,v).
 \]
In sum, the form $\hform$ must be symmetric or skew-symmetric and the sign $\spit$ records which case occurs.

\subsection{}  \label{inner-auts}
At a few places below, it is more  convenient to work with a variant $\ta'$ of the initial $\ta$. In each case, 
\[
     \theta' = \text{Int}\,(h) \circ \ta, 
\]
for some $h \in G$ where  $\text{Int}\,(h)$ denotes the inner automorphism $g \mapsto hgh^{-1}$ of $G$. 
We clearly also have $\pi^{\ta'} \simeq \pic$. 
For later use, we record the following simple relation between the signs $\spit$ and $\spip$.
Writing $\ompi$ for the central character of $\pi$, the element ${}^{\ta}h \,h$ is central in  $G$ and 
\begin{equation}  \label{sign-thetap}
 \spip =  \spit \, \ompi({}^{\ta}h \, h). 
\end{equation}

To check this, note first that $\ta' \circ \ta' = \text{Int}\, (h \,{}^{\ta}h)$. Since $\ta'$ has order at most two, $h \, {}^{\ta}h$ is in the center of $G$
and thus  $h \, {}^{\ta}h  =  {}^{\ta}h\, h$. 
Let $\hform$ be any nonzero form on $V \times V$ with the invariance property (\ref{theta-invariance}) and 
define a bilinear form $\sform$ on $V \times  V$ by 
\[
     [v_1, v_2]  =  (v_1, \pi(h^{-1})\,v_2), \,\,\, \quad v_1, v_2 \in V. 
\]
It is immediate that  
\[  
    [\pi(g) v_1, \pi^{\theta'}(g) v_2]  =  [v_1, v_2],  
  \quad  \,\,g \in G, \,\,\, \,\, v_1, v_2 \in V.
\]
The relation (\ref{sign-thetap}) now follows from the computation    
\begin{align*} 
      [v_1, v_2]  &=  \spit \, (\pi(h^{-1})\,v_2, v_1) \\
                         &=  \spit \, (v_2, \pi({}^{\ta}h) v_1) \\
                         &=   \spit \, \ompi({}^{\ta}h \, h) \,( v_2, \pi(h^{-1}) v_1) \\
                         &=    \spit \, \ompi({}^{\ta}h\, h)\,[v_2, v_1],   \quad \,\,\, \,\, v_1, v_2 \in V.     
\end{align*}

\section{Involutions and duals: examples}  \label{examples}

For the remainder of the paper,  $F$ is a non-archimedean local field and the 
group $G$ will always  be the $F$-points of  some reductive (usually connected) $F$-group. 
To provide examples of phenomena that we discuss in a general language in later sections, we now list  
several such groups $G$ and associated involutary automorphisms $\theta$ such that 
$\pit \simeq \pic$, for {\it all} irreducible smooth representations $\pi$ of $G$. 

\subsection{} \label{GK}
We first restate Gelfand-Kazhdan's result (\cite{GK} Theorem 2).
Let  $n$ be a positive integer and set $G = {\rm GL}_n(F)$. For $g \in G$, we put ${}^{\ta}g = {}^{\top}g^{-1}$ (where $\top$ denotes
\emph{transpose}). 

\begin{Thm}
Let $\pi$ be a smooth irreducible representation of $G$.  Then    
$\pit \simeq \pic$.  
\end{Thm} 

\noindent
We show in \S\ref{mainapp} and again in \S\ref{mainapp2} that the associated twisted signs are invariably one. 

\subsection{}  \label{MSR}
Next let $D$ be a quaternion division algebra over $F$.  
Let $n$ be a positive integer and set $G = {\rm GL}_n(D)$. 
Write $a \mapsto \overline{a}$ for the canonical involution or quaternion conjugation on $D$ 
(see, for example, \cite{Knus} p. 26). 
If $g = (a_{ij}) \in G$, we put  $\overline{g} = (\overline{a}_{ij})$ and set 
${}^{\ta}g = {}^{\top}\overline{g}^{-1}$. The map $\ta$ is an involution on $G$.  
The exact analogue in this setting of the Gelfand-Kazhdan theorem was proved by Mui\'{c} and Savin \cite{MS}
(in characteristic zero) and in a more elementary fashion by Raghuram  \cite{R} (in arbitrary characteristic).

\begin{Thm}
Let $\pi$ be a smooth irreducible representation of $G$.  Then    
$\pit \simeq \pic$.  
\end{Thm} 

\noindent
In this case, the associated twisted sign is controlled by the central character. More precisely, 
\[
    \spit = \ompi(-1), 
    \]
for any irreducible smooth representation $\pi$ of $G$ where $\ompi$ denotes the central character of $\pi$.
The proof is considerably more involved than the case of ${\rm GL}_n(F)$ \cite{RS}.     

\begin{Rmk}
For completeness, we record a converse observation. 
Let $\mathcal{D}$ be a finite-dimensional central division algebra over $F$ and set $G = {\rm GL}_n(\mathcal{D})$.  Suppose there is an 
involution $\ta$ on $G$ such that $\pit \simeq \pic$, for all irreducible smooth representations $\pi$ of $G$.  
By exploiting the structure of the automorphism group of $G$ \cite{D},  it is straightforward to show  the following (see \cite{RS} for details): 
 \begin{enumerate}[(a)]
\item 
$\mathcal{D} = F$ or $\mathcal{D} = D$;

\item
$\ta$ is unique up to composition with an inner automorphism. 

\end{enumerate} 
In other words, the only twisted signs that arise for the full smooth dual of the unit group of a central simple algebra over a non-Archimedean local 
field are those described above.
\end{Rmk}

\subsection{}    \label{MVW}
We recall a broad class of examples from \cite{MVW}. Suppose $F$ does not have characteristic two and let
$F'$ be $F$ or a quadratic extension of $F$. For $\epsilon = \pm1$, let  $\vform$ be a non-degenerate $\epsilon$-hermitian form on 
a finite-dimensional $F'$-vector space $V$.  Thus $\vform$ is orthogonal or symplectic in the case $F'=F$. 
We write $G$ for the associated isometry group: 
\[
  G = \{ g \in \text{Aut}_{F'}(V): \langle g u, g v \rangle = \langle u, v \rangle, \,\,\,\forall \, u, v \in V \}.  
\]
There is an element $\ta \in \text{Aut}_{F'}(V)$ with $\ta^2 = 1_V$, the identity map on $V$, such that 
\[
          \langle \ta u, \ta v \rangle  = \langle v, u \rangle, \quad \forall \, u, v \in V.
          \]
We also write $\ta$ for the induced involution on $G$ given by $g \mapsto \ta g \ta^{-1}$ ($g \in G$). 
We restate \cite{MVW} Chap.~4~II.1.

\begin{Thm} 
The involution $\ta$ satisfies $\pit \simeq \pic$ for all irreducible smooth representations $\pi$ of $G$.
\end{Thm}

\section{Casselman's pairing} 
We work in this section and the next in a general setting. Thus let $G$ be the group of $F$-points of a connected reductive $F$-group.  
As above, let $\ta$ be an involution on $G$ and let $\pv$ be an irreducible smooth representation of $G$ such that 
$\pit \simeq \pic$.  We use Casselman's pairing to show that  under suitable 
(widely applicable) conditions the sign $\spit$ can be studied via certain Jacquet modules of $\pi$.

\subsection{}  \label{setup}
We fix a maximal $F$-split torus $A$ in $G$ and write $\Phi = \Phi(A,G)$ for the set of roots of $A$ in $G$. 
(Here and throughout the paper we follow standard abuses of notation in failing to distinguish between algebraic $F$-groups and their groups of 
$F$-points.) We also fix a minimal $F$-parabolic subgroup $P_{\text{min}}$ 
of $G$ containing $A$.  
The group $P_{\text{min}}$ corresponds to a positive system $\Phi^+$ in $\Phi$.  We write $\De$ for  the unique
simple system contained in $\Phi^+$. 
Further, let $P$ be a standard parabolic subgroup of $G$ (i.e., $P \supset P_{\text{min}}$) and write
$M$ for the standard Levi component of $P$ (i.e., the 
unique Levi component containing $A$). Thus  
$P = M \ltimes N$ where $N$ denotes the unipotent radical of $P$. 
Note, for later use,  that  $M$ corresponds to a (unique) subset $\Ta$ of $\De$
in such a way that the maximal $F$-split torus $A_M$ in the center of $M$ is the identity component of  
$\bigcap_{\alpha \in \Ta} \ker \,\alpha$. We write $\overline{P}$ for the $M$-opposite of $P$ and $\overline{N}$ 
for the unipotent radical of $\overline{P}$.    

We impose three assumptions on the involution $\ta$: 

\smallskip

\begin{enumerate}[(a)]

\item
$\ta$ is an automorphism of $G$ as an algebraic group; 

\medskip

\item
$\ta$ preserves $A$ so that $\ta | A$ defines an involutory automorphism of the $F$-split torus $A$;   

\medskip

\item
${}^{\ta} N = \overline{N}$. 

\end{enumerate}

\begin{Rmk}
The involutions of  \S\ref{GK} and \S\ref{MSR}  clearly satisfy these conditions with respect to the standard  
maximal $F$-split tori consisting of (suitable) diagonal matrices and with respect to all standard parabolic subgroups of block upper triangular 
matrices.  Our framework also accommodates (in the connected case) the examples of \S\ref{MVW}.
\end{Rmk}

By (a) and (b), $\ta$ induces an involution $\alpha \mapsto \ta(\alpha)$ of $\Phi$  
where $\ta(\alpha)(a) = \alpha({}^{\ta} a)$, for $a \in A$.  
The set of roots of $A$ in $\text{Lie}(N)$ (respectively $\text{Lie}(\overline{N})$) is exactly 
 $\Phi^+ \setminus \mathbb{Z} \Ta$  (respectively $\Phi^- \setminus \mathbb{Z} \Ta$ where, as usual, $\Phi^- = - \Phi^+$). 
 Thus (c) is equivalent to 
\begin{equation} \label{roots}
   \ta(\Phi^+ \setminus \mathbb{Z} \Ta)  = \Phi^- \setminus \mathbb{Z} \Ta.  
\end{equation}
It follows that 
\[
   \ta(\Phi \cap \mathbb{Z} \Ta) = \Phi \cap \mathbb{Z} \Ta,  
\]
and hence ${}^{\ta}A_M = A_M$. Since $M = C_G(A_M)$, we see also that 
\[
{}^{\ta}M  =  M.
\]

\subsection{}

We write $(\pi_N, V_N)$ for the 
normalized Jacquet module of $\pv$ with respect to $P$. 
Explicitly, $V_N = V / V(N)$ where
\[
  V(N) = \langle \, \pi(n) v - v : v \in V, \, n \in N \, \rangle.
\]
For $v \in V$, we set $\overline{v} = v + V(N)$, so that, for $m \in M$,  
\[
  \pi_N(m) \, \overline{v} = \depm(m) \, \overline{\pi(m)\,v },   
\]
where $\dep$ denotes the modulus character of $P$.
We sometimes write $V^\ta$ for the space $V$ when viewed
as the space of the representation $\pit$ and  
use similar notation in other settings. 
In particular,
$(V^\ta)_{\overline{N}}$ denotes  the space of the Jacquet module of $\pit$ 
relative to $\overline{P}$ on which $M$ acts via 
$(\pit)_{\overline{N}}$. Explicitly, if 
$\overline{v} = v + V^\ta(\overline{N})$, for $v \in V^\ta$, then 
\[
(\pit)_{\overline{N}}(m)  
 \overline{v}  =   \depom(m) \, \overline{\pi^\ta(m)v },  
\] 
for $m \in M$.

We fix a non-degenerate $G$-invariant bilinear form $\hform$ on $V \times V^\ta$ as in (\ref{theta-invariance}).
We record what Casselman's pairing \cite[\S4]{Cass} yields in this   
setting. For $\epsilon > 0$, recall that 
\[
A_M (\epsilon)    =  \{ \, a \in A_M : ||\alpha(a)|| < \epsilon, \, 
           \forall \,\, \alpha \in \De \setminus \Ta \, \}, 
\]
where $\Ta$ corresponds to $M$ as in \S \ref{setup} and 
where $||\,\,||$ denotes the normalized absolute value on $F$.

\begin{cass}
There is a unique pairing  
$\jform:V_{\mathstrut N} \times (V^\ta)_{\mathstrut \overline{N}} \to \cx$ 
with the following property:   
given $u, v \in V$, there is an $\epsilon > 0$ such that
\[
    (\pi(a)u, v)  = 
           \depp (a) \,(\pi_N(a) \overline{u}, \overline{v})_N, \quad \,\, \forall \,a \in A_M (\epsilon). 
\]
The pairing $\jform$ is non-degenerate and $M$-invariant. 
\end{cass}

\begin{Rmk}
Casselman works with the canonical $G$-invariant pairing between 
an admissible smooth representation and its smooth dual.  
The proofs in \cite{Cass} carry over immediately to our setting  as the only necessary properties of the pairing are $G$-invariance and 
non-degeneracy. Alternative references are \cite{Bu} and \cite{Ren}~VI.9.6 which rework Casselman's construction (using ideas of Bernstein)  
and remove the admissibility assumption.  
\end{Rmk}

\subsection{}
We observe next that $(\pi^\ta)_{\overline{N}} = (\pi_N)^\ta$. 
Indeed,  
\begin{align*}
 V^\ta (\overline{N}) &= \langle \, \pi^\ta(\overline{n}) v - v :  v \in V, \, \overline{n} \in \overline{N} \, \rangle \\
                                    &=  \langle \, \pi({}^{\ta}\overline{n}) v - v : v \in V, \, \overline{n} \in \overline{N} \, \rangle \\
                                    &=  \langle \, \pi(n) v - v : v \in V, \, n \in N \, \rangle \quad (\text{by assumption (c) of \ref{setup}}) \\
                                    & =  V(N), 
\end{align*}
and so $(\pi^\ta)_{\overline{N}}$ and $(\pi_N)^\ta$ share the same underlying space. 

For $m \in M$ and $v \in V$, we have 
\begin{align*}
      (\pi^{\ta})_{\overline{N}} \, \overline{v}  &=   \depom(m)\,\overline{\pi^\ta(m)\, v}, \\
 (\pi_N)^\ta (\overline{v})  &= \deppo({}^{\ta}m) \, \overline{\pi({}^{\ta}m)\,v}. 
\end{align*}
Thus $(\pi^\ta)_{\overline{N}} = (\pi_N)^\ta$ provided  $ \depo(m) = \dep({}^{\ta}m)$, or equivalently  
\begin{equation}  \label{modtheta}
    \depo({}^{\ta} m) = \dep(m), \quad \,\,\, m \in M. 
\end{equation}
To check this, let $K$ be any compact open subgroup of $N$ and set $K^m = m^{-1} K m $. Then
\begin{equation}  \label{modp}
 \dep(m) =  [K: K^m],
 \end{equation} 
 where we have used the generalized index notation
\[
[K:  K^m ]  =  \frac{ [K:K \cap K^m] } {  [ K^m: K \cap K^m ] }.
\]
Applying $\ta$,  we obtain 
\[
[K: K^m]  = [ {}^{\ta}K: ({}^{\ta} K)^{   {}^{\ta} m }  ],  
\]
and so (\ref{modtheta}) follows (using (\ref{modp}) and its analogue for $\overline{P}$).
 
We can therefore view Casselman's pairing as a pairing on 
$V_N \times (V_N)^\ta$, and will do this from now on. For emphasis, we 
rewrite its defining properties in these terms.

\begin{Prop}
There is a unique pairing  
$\jform:V_N \times (V_N)^\ta \to \cx$ with 
the following property:   
given $u, v \in V$, there is an $\epsilon > 0$ such that, for any 
$a \in A_M (\epsilon)$, 
\[
    (\pi(a)u, v)  = 
           \depp (a) \,(\pi_N(a) \overline{u}, \overline{v})_N. 
\]
The pairing $\jform$ is non-degenerate and $M$-invariant:  
for $m \in M$ and $u, v \in V$, 
\[ 
    (\pi_N(m) \overline{u}, \pi_N({}^{\ta}m) \overline{v} )_N  = 
                    (\overline{u}, \overline{v})_N.
\] 
\end{Prop}

\subsection{}
Using uniqueness of Casselman's pairing, we now show that the form $\jform$ inherits the symmetry 
or skew-symmetry of the original form $\hform$.  

To see this, we claim first that the map $a \mapsto {}^{\ta}a^{-1}$ preserves $A_M(\epsilon)$ for $0 < \epsilon < 1$ . Indeed, 
if $\alpha \in \De \setminus \Ta$ then $\ta(\alpha) \in \Phi^- \setminus \mathbb{Z} \Ta$.  Writing $-\ta(\alpha)$ as a non-negative 
integral combination of roots in $\De$, we see that  some element of $\De \setminus \Ta$ must occur with a nonzero (hence positive) 
coefficient,   whence the claim. 

Now let $u, v \in V$ and set 
\[
   (\ou, \ov)_N'  =  (\ov, \ou)_N.
\]
Choose $\epsilon$ with $0 <  \epsilon < 1$ such that
\[
    (\pi(a)u, v)  = 
           \depp (a) \,(\pi_N(a) \ou, \ov)_N, \quad \,\, \forall \, a \in A_M (\epsilon).     
\]
Thus  
\begin{align*}    
     (\pi(a) u, v)  &=  \depp(a) \, (\ou, \pi_N({}^{\ta}a^{-1})\ov)_N \\     
             &=  \depp(a) \, (\pi_N({}^{\ta}a^{-1})\ov, \ou)_N',   \quad \,\, \forall \, a \in A_M (\epsilon).     
\end{align*} 
We also have
\begin{align*}    
\quad (\pi(a) u, v)  &=  (u, \pi({}^{\ta}a^{-1})v) \\
                                     &=  \spit \, (\pi({}^{\ta}a^{-1})v, u),  \quad \,\, \forall \, a \in A_M (\epsilon),  
\end{align*}
and hence  
\begin{equation} \label{keycass}
 \spit \, (\pi({}^{\ta}a^{-1})v, u)   =   \depp(a)\, (\pi_N({}^{\ta}a^{-1})\ov, \ou)_N',    \quad \,\, \forall \, a \in A_M (\epsilon).     
\end{equation} 
For any $a \in A_M$,  
\[
         \dep({}^{\ta}a^{-1})  = \depo(a^{-1})  = \dep(a).
\]         
Here the first equality is given by (\ref{modtheta}). The second follows, for example, from the formula                                
\[
    \dep(a) = ||\det (\text{Ad}\,a : \text{Lie}(N))||,  
\]
(valid for $a \in A$) and its analogue for $\overline{P}$ once one notes that the roots of $A$  in $\text{Lie}(\overline{N})$ are
the negatives (with multiplicity) of the roots of $A$ in $\text{Lie}(N)$. 
Since $a \mapsto {}^{\ta}a^{-1}$ preserves $A_M(\epsilon)$, we can rewrite (\ref{keycass}) as
\[
\spit \, (\pi(a)v, u)   =   \depp(a)\, (\pi_N(a)\ov, \ou)_N',    \quad \,\, \forall \, a \in A_M (\epsilon).     
\]
Thus, by uniqueness of Casselman's pairing, 
\[
  \spit \, \jform '  =  \jform, 
\]
i.e., the form $\jform'$ admits a sign $\spin$ and 
\begin{equation}  \label{descent}
      \spin = \spit.  
\end{equation}

\subsection{} \label{mult-one}
We record a simple but crucial consequence  of the preceding discussion.  

\begin{Cor}
Let $\tau$ be an irreducible subrepresentation of $\pi_N$ that occurs with multiplicity one (as a composition factor of $\pi_N$) and suppose that 
$\tau^\ta  \simeq \tc$. Then 
\[
    \spit = \sptn. 
\]
\end{Cor}

\begin{proof}  
Writing $Y$ for the space of the subrepresentation $\tau$ of $\pi_N$, the $M$-map  
\[
 x \mapsto (x,-)_N\, | \,Y: (\pi_N)^\ta \to \tau\spcheck 
\]
is nonzero. Thus there is an irreducible subquotient 
$\rho$ of $\pi_N$ such that  $\rho^\ta \simeq \tc$ and so $\rho \simeq \tau$. Hence
$\rho=\tau$  and  $\jform \mid Y \times Y$ is nonzero. Therefore $\sptn = \spin$. The result now follows from (\ref{descent}).
\end{proof}

\section{Reduction to tempered case} 
We continue in the setting and with the notation of the previous section. 
In particular,  $\pi$ is an irreducible smooth representation of $G$ and $\ta$ is an involution on $G$ such that $\pit \simeq \pic$. Further,  
$\ta$ remains subject to the assumptions of  \S\ref{setup} for any standard parabolic subgroup under discussion.
The Langlands classification attaches  a pair $(M,\tau)$ to $\pi$ where $M$ is a standard Levi subgroup of $G$ and $\tau$ is an irreducible
tempered representation of $M$.  We recall some of the details below.
Using the strong uniqueness properties of the classification and certain related properties, we show that 
$\tat  \simeq \tc$ and that $\spit = \sptn$. Thus the problem of  determining  
$\ta$-twisted signs reduces (under our assumptions) to the case of tempered representations.

\subsection{}
We first  review a small amount of detail concerning an ingredient of the Langlands classification (following for the most part the discussion in
 \cite{MW} \S1.1). 

All tensor products below are over $\mathbb{Z}$. 
For any algebraic group $L$ defined over $F$, we set $X(L) =  \text{Hom}_F(L, \mathbb{G}_m)$,  the group of 
$F$-rational characters of the algebraic group $L$ and put $\fals = X(L) \otimes  \mathbb{R}$.  When $L = A$, our fixed maximal
$F$-split torus in $G$, we simply write $\fas = X(A) \otimes \mathbb{R}$. 

Let $M$ be a  standard Levi subgroup of $G$. As in \S\ref{setup}, $M$ corresponds to a subset $\Theta$ of our fixed simple system $\Delta$ 
in the set of roots $\Phi$ of $A$ in $G$. Let $\rpos$ denote  the multiplicative group of positive real numbers.  We identify $\fams$ with
$\text{Hom}(M,\rpos)$, the group of continuous homomorphisms from $M$ to $\rpos$, via the isomorphism 
\begin{equation}  \label{rpchar}
\chi \otimes r \longmapsto (m \mapsto || \chi(m) ||^r): \fams \overset{\simeq}{\longrightarrow}  \text{Hom}(M,\rpos).
\end{equation}

The restriction map $\chi \mapsto \chi \, | \, A:X(M) \to X(A)$ is injective and so induces a canonical  injection from 
$\fams$ to $\fas$ via which we view $\fams$ as a subspace of $\fas$.  As usual, we set $X_\ast(A)= \text{Hom}(\mathbb{G}_m, A)$, the lattice of
one-parameter subgroups of $A$.  Associated to each $\alpha \in \Phi$
is a certain element $\alpha\spcheck \in X_\ast(A)$, the coroot corresponding to $\alpha$. The canonical 
pairing between $X_\ast(A)$ and $X(A)$ extends by  $\mathbb{R}$-linearity to a non-degenerate pairing  $\vform$ between 
$\mathfrak{a} = X_\ast(A) \otimes \mathbb{R}$ and $\fas$. In these terms, 
\begin{equation} \label{coroots}
    \fams = \{ \nu \in \fas :  \langle \alpha\spcheck,  \nu \rangle = 0, \,\,\forall \, \alpha \in \Theta \}.
\end{equation} 
Let
\[
   \fan = \{ \nu \in \fams : \langle \alpha\spcheck, \nu \rangle <  0, \,\,\forall \, \alpha \in \Delta \setminus \Theta \}.
\]
We may view the elements of $\fan$ as certain linear characters on $M$ via (\ref{rpchar}) and will do this without comment below.

By hypothesis, the involution $\ta$ preserves $A$ . There is  an induced action of $\ta$ on $X(A)$ (via $\ta(\chi) = \chi \circ \ta$ for $\chi \in X(A)$)
and a dual action on $X_\ast(A)$ and these extend to actions on $\fas$ and $\mathfrak{a}$. 
Note that the pairing $\vform:\mathfrak{a} \times \fas \to \mathbb{R}$ is $\ta$-invariant (since $\ta^2 = 1$).  

\begin{Lem}  \label{tafan}
The map $\nu \mapsto - \ta(\nu): \fas \to \fas $ preserves $\fan$ (i.e., it maps $\fan$ to itself).
\end{Lem}

\begin{proof}
Let $\alpha \in \Delta \setminus \Theta$. By (\ref{roots}), $\ta(\Delta \setminus \Theta) \subset \Phi^- \setminus \mathbb{Z}\Theta$. Thus  
if we write $\ta(\alpha)$ as a linear combination of elements of $\Delta$, then the coefficient of some element of $\Delta \setminus \Theta$ 
must  be negative. The statement is now an immediate consequence of (\ref{coroots}) using the $\ta$-invariance of $\vform$ and the identity 
$\ta(\alpha\spcheck)  = \ta(\alpha)\spcheck$, for $\alpha \in \Phi$.
\end{proof}
\noindent
We invariably use multiplicative notation when viewing the elements of $\fan$ as characters on $M$ 
and so the lemma then says that $\nu \mapsto (\nu^\ta{})^{-1}$ preserves $\fan$.

\subsection{} \label{reduction}
The Langlands classification uniquely attaches a certain triple $(P, \tau, \nu)$  to $\pi$. This consists of a standard parabolic subgroup $P$ of $G$,
an irreducible tempered representation $\tau$ of the standard Levi component $M$ of $P$ (up to equivalence), 
and an element $\nu \in \fan$. 
The representation $\pi$ is then the unique irreducible subrepresentation of the normalized induced representation $\indp (\tau \nu)$. 
We note some additional (closely related) properties of the classification.

\smallskip

\begin{enumerate}[(a)] 
\item 
$\pi$ is also the unique irreducible quotient  of $\indpb (\tau \nu)$. (See \cite{BW}~XI~ 2.7 but observe that the roles
of $P$ and $\overline{P}$ are reversed.) 
\smallskip
\item
$\tau \nu$ occurs as a subrepresentation of the Jacquet module $\pi_N$ and has multiplicity one as a composition factor of $\pi_N$.
In fact, a stronger statement holds by \cite{BJ}~\S5.
\end{enumerate}

\begin{Prop}
Let $\pi$ be an irreducible smooth representation of $G$ such that $\pit \simeq \pic$. Suppose the Langlands classification attaches the 
triple $(P, \tau, \nu)$ to $\pi$. Then $\tau^\ta  \simeq \tc$ and $\spit = \sptn$.
\end{Prop}

\begin{proof}
Applying $\ta$ to the inclusion $\pi \hookrightarrow  \indp (\tau \nu)$ and using ${}^{\ta}P = \overline{P}$, we see that 
\[
     \pit \hookrightarrow  \indp (\tau \nu){\,^\ta} \simeq \indpb (\tau^\ta \nu^\ta).
\]
Thus $\pic$ embeds in $\indpb (\tau^\ta \nu^\ta)$. Dualizing, it follows that $\pi$ is a quotient of 
\[
   \indpb (\tau^\ta \nu^\ta)\spcheck \simeq \indpb (\tau^\ta {}\spcheck  (\nu^\ta){}^{-1}).
     \]
Clearly,  $\tau^\ta {}\spcheck$ is an irreducible tempered representation of $M$. By Lemma \ref{tafan},   $(\nu^\ta){}^{-1}  \in \fan$. 
Using (a) and uniqueness of the triple attached to $\pi$, we see that $\tau^\ta {}\spcheck \simeq \tau$ and $(\nu^\ta){}^{-1} = \nu$, or equivalently 
\[  
   \tau^\ta  \simeq \tc, \,\,\, \nu^\ta = \nu^{-1}.
   \]
Of course, this gives $(\tau \nu)^\ta \simeq (\tau \nu)\spcheck$.  We also have the trivial identity     
$\varepsilon_\ta(\tau \nu) = \sptn$. The result now follows from (b) above and Corollary \ref{mult-one} (applied to $\tau \nu$). 
\end{proof}

\section{Twisted signs for general linear groups I}   \label{mainapp}

Let $n$ be a positive integer and set $G = {\rm GL}_n(F)$. 
As in \S\ref{GK}, put ${}^{\ta}g =  {}^{\top}g^{-1}$ (where $\top$ denotes
\emph{transpose}).  
Let $\pi$ be a smooth irreducible representation of $G$. 
We show that the associated sign is invariably one.  
  
\begin{Thm}
Let $\pi$ be an irreducible smooth representation of $G$. Then $\spit = 1$.  
\end{Thm}
\begin{Rmk}
This has been proved in many instances  by Vinroot (see \S6 of \cite{Vin} and the remarks following Cor.~3.1 of \cite{LV}). 
Our approach reduces us to the case of generic representations which is covered by his methods.  
In place of an appeal to \cite{Vin}, however, we complete the argument in a slightly different way. 
\end{Rmk}

By Proposition \ref{reduction}, we have only to consider tempered $\pi$. 
As every irreducible tempered representation of $G$ is generic \cite{Ze}, it suffices therefore to prove the following.

\begin{Prop}
Let $\pi$ be an irreducible generic representation of $G$. 
Then $\spit = 1$.
\end{Prop}
\begin{proof}
Since $\pi$ is generic, the main result of 
\cite{cond} (see also  \cite{J, Mat}) says 
there is a nonnegative integer $c$ (in fact, a unique such integer) 
such that 
$\dim \pi^K = 1$ where $K=K_c$  is the compact open subgroup
given in block matrix form by 
\[
K  =  \begin{bmatrix}
                     {\rm GL}_{n-1}(\mathfrak{o}) &\mathfrak{o} \\
                        \mathfrak{p}^c   & 1 + \mathfrak{p}^c
          \end{bmatrix}.  
\]
(Here $\mathfrak{o}$ is the valuation ring of 
$F$ and $\mathfrak{p}$ is its unique maximal ideal, and we   
interpret $K$ as ${\rm GL}_n(\mathfrak{o})$ in the case $c = 0$).   
Of course, 
\[
{}^{\ta}K  =  \begin{bmatrix}
                     {\rm GL}_{n-1}(\mathfrak{o}) &\mathfrak{p}^c \\   
                              \mathfrak{o} &1 + \mathfrak{p}^c
          \end{bmatrix}.
\]

We fix a uniformizer $\varpi$ and set $a  =   \text{diag}\,(1, \ldots, 1, \varpi^c)$ and 
\[
        \theta' = \text{Int}\,(a) \circ \ta. 
\]
It is immediate that $ \theta'$ is again an involution. By (\ref{sign-thetap}), 
\[
     \spip = \spit.
   \]
To complete the proof, we show that $\spip = 1$. 
Indeed, the involution $ \theta'$ was constructed so that ${}^{\theta'}K = K$, 
and hence      
\begin{align*}
   \dim \, (\pi^{\theta'})^K  &=  \dim \, \pi^{({}^{\theta'}K)} \\
                               & =  \dim \, \pi^K \\
                               & =  1.
\end{align*} 
A non-degenerate $G$-invariant form on $V \times V^{\ta'}$ remains non-degenerate on restriction to  
$V^K \times (V^{\theta'}){}^{K}$.  Since $V^K$ is a line, it follows that $\spip = 1$. 
\end{proof}

\section{Twisted signs for general linear groups II}   \label{mainapp2}
We continue with the notation of the preceding section. Thus $G = {\rm GL}_n(F)$ and 
${}^{\ta}g =  {}^{\top}g^{-1}$, for $g \in G$. We offer a second proof of Theorem~\ref{mainapp} 
as a simple illustration of the way in which ordinary and twisted signs are closely intertwined.
As noted in \S\ref{mainapp}, it suffices by Proposition~\ref{reduction} to establish the following special case.

\begin{Thm}
Suppose $\pi$ is a smooth irreducible tempered representation of $G$. Then $\spit= 1$. 
\end{Thm} 

The strategy of the proof is to view $G$ as the Siegel Levi subgroup of a suitable split classical group $\widetilde{G}$ 
and to use Proposition~\ref{reduction} to relate the twisted sign of $\pi$ and the ordinary sign of an irreducible self-dual representation 
$\widetilde{\pi}$ of $\widetilde{G}$ (obtained from a certain twist of $\pi$ via parabolic induction). 
The representation $\widetilde{\pi}$ is generic and so its sign is known by a general result of Dipendra Prasad \cite{P1}.

\begin{proof}
Consider the quadratic form in $2n+1$ variables given by
\[
  q(\xi_1, \ldots, \xi_n, \gamma, \eta_1, \ldots, \eta_n)  = \sum_{i=1}^n \xi_i \eta_i + \gamma^2.
  \]
We take $\widetilde{G}$ to be the $F$-points of the identity component of the corresponding orthogonal group.   
(Thus  $\widetilde{G}$ is the $F$-points of the split group ${\rm SO}_{2n+1}$ when the characteristic of $F$ is not two. In  
the case of characteristic two, the isometry group of $q$ is already connected.) 
We realize $\widetilde{G}$ as a group of invertible $(2n+1) \times (2n+1)$ matrices over $F$ in the obvious way. 
The diagonal elements  
\[
  \text{diag}(a_1, \ldots, a_n, 1, a_1^{-1}, \ldots, a_n^{-1}) 
  \]
then form (the $F$-points of) a maximal split torus.  Writing $\epsilon_i$ for the rational character
 \[
  \text{diag}(a_1, \ldots, a_n, 1, a_1^{-1}, \ldots, a_n^{-1})  \longmapsto a_i
  \]
and using additive notation, we take $\{\epsilon_1-\epsilon_2, \ldots, \epsilon_{n-1} - \epsilon_n, \epsilon_n \}$ as our fixed simple system  
in the set of roots of $A$ in $\widetilde{G}$. 
The subset $\Theta = \{\epsilon_1-\epsilon_2, \ldots, \epsilon_{n-1} - \epsilon_n\}$ corresponds to the standard Siegel parabolic 
subgroup $P$ of $\widetilde{G}$. This consists of the block upper-triangular matrices
\[
       \begin{bmatrix} 
      g &\ast & \ast \\
      0 & 1 & \ast \\
      0 & 0 &   {}^{\top}g^{-1}
      \end{bmatrix}, \quad g \in G.
 \]     
The group 
\[
M = \{ \begin{bmatrix} 
      g &0 & 0 \\
      0 & 1 & 0 \\
      0 & 0 &   {}^{\top}g^{-1}
      \end{bmatrix} : g \in G \}.
\]
is the corresponding Levi component  of $P$.

We use the obvious isomorphism 
\[
g \mapsto \begin{bmatrix} 
      g &0 & 0 \\
      0 & 1 & 0 \\
      0 & 0 &   {}^{\top}g^{-1}
      \end{bmatrix}:G \to M
      \]
to view $\pi$ as an irreducible tempered representation of $M$. 

Consider the element $n \in \widetilde{G}$ given by
\[
n = \begin{bmatrix} 
          0  & 0  &I_n \\
          0  & (-1)^n & 0 \\
          I_n  &0   &0 
          \end{bmatrix}
\]
and set $\widetilde{\theta} = \text{Int} (n)$. Since $n^2 = 1$, the automorphism $\widetilde{\theta}$ is clearly an involution on $\widetilde{G}$. 
Further, it is immediate that $\widetilde{\theta}$ satisfies the three conditions of \S\ref{setup}. Note  
\[
   \widetilde{\theta}\bigl( \begin{bmatrix} 
      g &0 & 0 \\
      0 & 1 & 0 \\
      0 & 0 &   {}^{\top}g^{-1}
      \end{bmatrix} \bigr)   = 
      \begin{bmatrix} 
      {}^{\top}g^{-1} &0 & 0 \\
      0 & 1 & 0 \\
      0 & 0 &   g
      \end{bmatrix}, \quad g \in G. 
      \]
In other words, $\widetilde{\theta} \, | \, M $ corresponds to $\theta$ under the above isomorphism      
between $G$ and $M$. 

As a very special case of Theorem~3.2 of \cite{Sauv}, there is a $\nu \in \fan$ such that the normalized induced representation 
$\widetilde{\pi} = \iota_P^{\widetilde{G}}(\pi \nu)$ is irreducible.  We need to observe that $\widetilde{\pi}$ is self-dual. For this, note  
${}^nP = \overline{P}$, the $M$-opposite of $P$, and ${}^{n}(\pi \nu) \simeq \pic \nu^{-1}$ (by Gelfand-Kazhdan's theorem). Thus 
\begin{align*}
 \iota_P^{\widetilde{G}}(\pi \nu)  &\simeq {}^{n}  \iota_P^{\widetilde{G}} (\pi \nu) \\
                                                         &\simeq  \iota_{\overline{P}}^{\widetilde{G}}(\pic \nu^{-1}) \\
                                                         &\simeq  \iota_{\overline{P}}^{\widetilde{G}}(\pi \nu)\spcheck. 
\end{align*}
It is a fundamental property of parabolic induction that the (multi-)set of composition factors of a parabolically induced representation depends only
on the inducing representation of the Levi subgroup, i.e., it is independent of the parabolic subgroup with the given Levi component. In particular, 
$\iota_{\overline{P}}^{\widetilde{G}}(\pi \nu)\simeq \iota_P^{\widetilde{G}}(\pi \nu)$ and thus $\widetilde{\pi} \simeq \widetilde{\pi}\spcheck$. 
(If $F$ does not have characteristic two,  we can instead appeal  to \cite{MVW} Chap.~4 II.1 (restated in \S\ref{MVW}) which gives that  
every smooth irreducible representation of $\widetilde{G}$ is self-dual.)

We clearly also have $\widetilde{\pi}^{\,\widetilde{\theta}} \simeq \widetilde{\pi}\spcheck$ (as $\widetilde{\theta}$ is an inner automorphism). 
By (\ref{sign-thetap}), 
\[
     \varepsilon_{}(\widetilde{\pi})  =   \varepsilon_{\mathstrut  \widetilde{\theta}}(\widetilde{\pi}).
\]
Further, by Proposition~\ref{reduction}, 
\[
       \varepsilon_{\mathstrut \widetilde{\theta}}(\widetilde{\pi}) =  \varepsilon_\theta(\pi).
\]
We now show that 
\begin{equation*}
 \varepsilon(\widetilde{\pi}) = 1.      \tag{$\ast$}
\end{equation*}
As noted in \S\ref{mainapp}, the representation $\pi$ is generic and thus $ \widetilde{\pi}$ is also generic 
(by Rodier's heredity theorem \cite{Rod}). By \cite{P1} page 448, the sign of an
irreducible generic representation of $\widetilde{G}$ is always one. This establishes ($\ast$) and so completes the proof. 
\end{proof}

\end{document}